\documentclass[12pt]{amsart}
\usepackage{amsfonts, amssymb, mathrsfs, amsmath}
\usepackage{float, fullpage, verbatim, bbm}
\usepackage[colorlinks=true]{hyperref}
\usepackage{breakurl}
\usepackage{mathtools}
\usepackage{multicol}
\usepackage{tikz}
\usetikzlibrary{calc}
\usepackage[english]{babel}
\usepackage{etoolbox}
\makeatletter
\let\ams@starttoc\@starttoc
\makeatother
\usepackage[parfill]{parskip}
\makeatletter
\let\@starttoc\ams@starttoc
\patchcmd{\@starttoc}{\makeatletter}{\makeatletter\parskip\z@}{}{}
\makeatother

\numberwithin{equation}{section}

\def\Z{{\mathbb{Z}}}
\def\K{{\mathbb{K}}}
\def\A{{\mathscr{A}}}
\def\B{{\mathscr{B}}}

\newcommand\CA{{\mathscr A}}

\newcommand\BBK{{\mathbb K}}

\newcommand\Der{{\operatorname{Der}}}

\newcommand\rank{\operatorname{rank}}

\newcommand{\bbmo}{\mathbbm{1}}

\numberwithin{equation}{section}

\theoremstyle{plain}
\newtheorem{lemma}[equation]{Lemma}
\newtheorem{theorem}[equation]{Theorem}

\newtheorem{problem}[equation]{Problem}
\newtheorem{corollary}[equation]{Corollary}
\newtheorem{proposition}[equation]{Proposition}
\theoremstyle{definition}
\newtheorem{defn}[equation]{Definition}

\newtheorem{example}[equation]{Example}

\begin{document}

\title{On Universal derivations for multiarrangements}

\author[Takuro Abe]{Takuro Abe}
\address
{Department of Mathematics, Rikkyo University, Tokyo 171-8501, Japan.}
\email{abetaku@rikkyo.ac.jp}

\author[Shota Maehara]{Shota Maehara}
\address
{Joint Graduate School of Mathematics for Innovation, Kyushu University, 744 Motooka Nishi-ku Fukuoka 819-0395, Japan}
\email{maehara.shota.027@s.kyushu-u.ac.jp}

\author[G.~R\"ohrle]{Gerhard R\"ohrle}
\address
{Fakult\"at f\"ur Mathematik,
	Ruhr-Universit\"at Bochum,
	D-44780 Bochum, Germany}
\email{gerhard.roehrle@rub.de}

\author[S.~Wiesner]{Sven Wiesner}
\address
{Fakult\"at f\"ur Mathematik,
	Ruhr-Universit\"at Bochum,
	D-44780 Bochum, Germany}
\email{sven.wiesner@rub.de}

\subjclass[2020]{Primary 52C35, 32S22, 14N20}

\keywords{
	Free arrangement, 
	Free multiarrangement,
	Universal derivation}

\allowdisplaybreaks
\maketitle

\begin{abstract}
The study of universal derivations for arbitrary multiarrangements and multiplicity functions was initiated by Abe, Röhrle, Stump, and Yoshinaga in \cite{ASRY} which focused on arrangements arising from (well-generated) reflection groups. In this paper we provide a criterion for determining whether a derivation is universal along with a characterization of universal derivations for arbitrary 2-multiarrangements. As an application we give descriptions of universal derivations for several multiarrangements, including the so-called deleted $A_3$ arrangement. This is the first known example of a non-reflection arrangement that admits a universal derivation distinct from the Euler derivation.
\end{abstract}


\section{Introduction}
The freeness of hyperplane arrangements has been a central topic in the theory of arrangements for several decades. Despite extensive research, determining whether a given arrangement is free remains a challenging problem. A major open question in this area is Terao’s conjecture which asserts that the freeness of an arrangement depends only on its combinatorial structure, namely its intersection lattice. A function $m:\A \rightarrow \Z_{\ge0}$ on an arrangement $\A$ is called a \emph{multiplicity} and a pair $(\A,m)$ a \emph{multiarrangement}. An important advancement in the study of freeness was made by Yoshinaga, who established a relationship between the freeness of a central arrangement $\A$ and that of a certain multiarrangement, arising as a restriction of $\A$ to one of its hyperplanes, the so called Ziegler restriction. Consequently, the theory of multiarrangements has become a significant area of investigation in its own right.

In the special case when $\A = \A(W)$ is the reflection arrangement of a finite Coxeter group $W$, the concept of universal vector fields was introduced by Yoshinaga in \cite{Y0} to construct bases for multi-Coxeter arrangements using affine connections. They were further explored in \cite{ATW} and \cite{AY} as a tool to study the structure of the module of logarithmic derivations $D(\A, m)$ for a given multiplicity function $m: \CA\to\mathbb{Z}_{\ge0}$. A definition was later provided by Wakamiko in \cite[Def.~2.2]{W10}. Note that universal vector fields are defined only in the setting of Coxeter arrangements $\A(W)$, where the $W$-action and $W$-invariance play a fundamental role in their construction. In \cite{ASRY}, Abe, Stump, Röhrle, and Yoshinaga extended the definition of universal vector fields to universal derivations for arbitrary arrangements. 

Let $V$ be an $\ell$-dimensional vector space over a field $\K$ of characteristic zero, $S=\K[x_1,\ldots,x_\ell]$ the coordinate ring of $V$, 
and $\Der_S=\oplus_{i=1}^\ell S \cdot \partial_{x_i}$ the module of 
$S$-regular derivations. For a given multiplicity $m$ on $\A$, the \emph{module of $(\A,m)$-derivations} $D(\A,m)$ is the central object of this study, see \eqref{def:DAm}. We define the trivial multiplicity 
$\bbmo$ on $\A$ by 
$\bbmo(H) = 1$ for each $H$ in $\A$. 

\begin{defn}\label{definiton: Universal Derivations}
Let  $(\A,m)$ be a multiarrangement with fixed multiplicity $m$ and $\theta\in D(\A,m+\bbmo)$ a homogeneous derivation. Then $\theta$ 
is said to be \emph{$m$-universal (for $\A$ or $(\A,m)$)} if the map
$$
\Phi_\theta:\Der_S \to D(\A,m), \varphi \mapsto \nabla_\varphi \theta$$
is an isomorphism of $S$-modules, where for $\theta=\sum_{i=1}^\ell f_i \partial_{x_i} \in \Der_S$,  $\nabla_\varphi \theta $ is defined by
$$
\nabla_\varphi \theta\coloneqq\sum_{i=1}^\ell \varphi(f_i) \partial_{x_i}.
$$ 
\end{defn}
Note that if both $\theta$ and $\varphi$ are homogeneous, then 
$\deg(\nabla_{\varphi}\theta)=\deg \varphi + \deg \theta - 1$.
In particular, if $\theta$ is $m$-universal for $\A$, then $\Phi_\theta$ maps the $S$-basis $\{\partial_{x_i}\mid 1\leq i\leq \ell\}$ of $\Der_S$ to an $S$-basis of $D(\A,m)$. Thus $D(\A,m)$ is free with exponents $\exp(\A,m)=(d,d,\dots,d)$, where $\deg\theta=d+1$.

This notion is motivated by the Euler derivation.
Suppose $\A$ is free and irreducible. Then the Euler derivation $\theta_E = \sum_{i=1}^\ell x_i\partial_{x_i}$ belongs to $ D(\A) = D(\A,0+\bbmo)$ and $\nabla_\varphi \theta_E = \varphi$ for all $\varphi\in \Der_S$, so $\Phi_{\theta_E}$ is the identity on $\Der_S = D(\A,0)$. So the Euler derivation is $0$-universal. 

Universal derivations can be utilized to determine the structure of not only the $S$-modules $D(\A,m+\bbmo)$ and $D(\A,m)$ but also of intermediate $S$-modules. Hence an $m$-universal derivation can play an important role in the investigation of derivation modules and free arrangements. 

Our first result is a characterization of universal derivations which can be considered as a generalization of \cite[Thm.~3.4]{AN}. The notion of criticality was originally defined by Ziegler in \cite{Z2} (see Definition \ref{definition: k-critical}). It is used to derive the following criterion for universal derivations.

\begin{theorem}\label{theorem: alternative universal criterion}
Let $\A$ be an irreducible $\ell$-arrangement and $m$ a multiplicity on $\A$. 
Assume that $\theta \in D(\A,m+\bbmo)$ is homogeneous with $\deg \theta = d+1$. Then the following are equivalent:
\begin{itemize}
\item[(1)]
$\theta$ is $m$-universal;
\item[(2)]
$(\A,m)$ is free with $\exp(\A,m)=(d,\ldots,d)$ 
and $D(\A,m+\bbmo)$ is $(d+1)$-critical.
\end{itemize}
\end{theorem}

Our second main result classifies universal derivations for multiarrangements of rank two. For the notion of a balanced multiplicity, see Definition \ref{definition: balanced}.

\begin{theorem}\label{theorem: rank two universal criterion}
Let $(\A,m)$ be an irreducible $2$-multiarrangement. Assume that $|\A|>3$ or 
$|\A|=3$ and $m$ is balanced. Then a homogeneous derivation $\theta \in D(\A,m+\bbmo)$ is $m$-universal if and only if $m+\bbmo$ is balanced, and $\exp(\A,m+\bbmo)=(\deg \theta,\deg \theta+|\A|-2)$.
\end{theorem}

The following example illustrates why we require $\theta\in D(\A,m+\bbmo)$ instead of merely allowing that $\theta$ belongs to $\Der_S$ in Theorem \ref{theorem: rank two universal criterion}. 
Here and subsequently, for a fixed $H\in\A$ we define the indicator multiplicity 
$\delta_H:\A\to \{0,1\}$  by 
    $$\delta_H(H')= \left\{
\begin{array}{rl}
0 & \mbox{if}\ H' \neq H,\\
1 & \mbox{if}\ H' = H.
\end{array}
\right.$$

\begin{example}
    Let $\A$ be the Coxeter arrangement of type $B_2$ with defining polynomial $$Q(\A)=xy(x-y)(x+y).$$
    Consider the 2-multiarrangment $(\A,m+\bbmo)$ given by 
    $$Q(\A,m+\bbmo)=x^3y^5(x-y)^2(x+y)^2.$$
    Then $\exp(\A,m+\bbmo)=(5,7)$ and $\exp(\A,m)=(4,4)$. 
    Moreover, $\exp(\A,m+\bbmo+\delta_H)=(6,7)$ for every $H$ in $\A$.
    Let $\theta\in D(\A,m+\bbmo)$ be homogeneous with $\deg\theta=5$. Note that $\theta$ is unique up to multiplication by a non-zero scalar and is $m$-universal. If however $\theta$ is only required to belong to $\Der_S$ instead of $D(\A,m+\bbmo)$, then every derivation of degree $5$ in $\Der_S$ satisfies the conditions in Theorem \ref{theorem: rank two universal criterion}, even if $\Phi_\theta$ is not an isomorphism.
\end{example}

We utilize Theorem \ref{theorem: alternative universal criterion} to give examples of universal derivations. This includes a classification of the latter on the deleted $A_3$ arrangement which marks the first existence result of that kind  for arrangements not stemming from a (well-generated) reflection group.

\begin{theorem}\label{theorem: Multieuler for deleted A3}
    Let $\A$ be the deleted $A_3$ arrangement with defining polynomial $$Q(\CA)=(y-z)y(x-y)x(x-z).$$
    Consider the multiarrangment $(\A,m+\bbmo)$ given by 
    $$Q(\A,m+\bbmo)=(y-z)^ay^b(x-y)^cx^d(x-z)^e.$$ There exists an $m$-universal $\theta\in D(\A,m+\bbmo)$ if and only if $c = a+e-1 = b+d-1$.
\end{theorem}

We also derive a classification of universal derivations for the Braid arrangement of rank three equipped with a supersolvable multiplicity. For the latter see Definition \ref{definition: supersolvable multi}.

\begin{theorem}\label{A3multiEuler}
Let $(\A,m+\bbmo)$ be a Coxeter multiarrangement of type $A_3$. If $(\A,m+\bbmo)$ is supersolvable, then there exists an $m$-universal $\theta \in D(\A,m+\bbmo)$ if and only if all of the inequalities in Theorem \ref{theorem: supersolvable exponents} are identities.
\end{theorem}

It is apparent from Theorem \ref{theorem: alternative universal criterion} that if $m\not \equiv 0$ there do not exist $m$-universal derivations for totally non-free arrangements. Our next result shows that a free multiarrangement $(\A,m)$ with $m\not \equiv 0$ need not admit an $m$-universal derivation in general.

\begin{theorem}\label{theorem: classification universal for X3}
Let $\A$ be the $X_3$-arrangement defined by 
$$ Q(\A)=xyz(x+y)(y+z)(x+z).$$
 Then $\A$ does not admit an $m$-universal derivation for any multiplicity $m\not \equiv 0$ on $\A$. 
\end{theorem}

An arrangement $\CA$ is called \emph{totally free} if $(\CA,m)$ is free for every multiplicity $m:\CA\to\mathbb{Z}_{\geq 0}$. While a free arrangement $\CA$ is $0$-universal with universal derivation $\theta_E\in D(\CA)$, our final result shows that freeness does not imply the existence of a universal derivation $\theta\neq\theta_E$. 

\begin{theorem}\label{theorem: freeness does not imply universal}
    Let $\CA$ be a totally free arrangement. Then there need not exist an $m$-universal derivation $\theta\neq\theta_E$ for any multiplicity $m\not \equiv 0$. 
\end{theorem}

The organization of this article is as follows. Section \ref{sect: preliminaries} contains preparatory material. In Section \ref{sect: main result proof} we prove Theorems \ref{theorem: alternative universal criterion} and \ref{theorem: rank two universal criterion}. 
In Section \ref{sect: examples} we give numerous new examples of $m$-universal derivations, mainly stemming from supersolvable multiarrangements.
\medskip

\textbf{Acknowledgements}. 
T.~Abe was partially supported by 
JSPS KAKENHI grant numbers 
JP23K17298 and 25H00399. 
S.~Maehara was supported by WISE program (MEXT) at Kyushu University. 
This work was also supported by DFG-Grant
RO 1072/21-1 (DFG Project number 494889912) to G.~R\"ohrle.


\section{Preliminaries}\label{sect: preliminaries}
A \emph{hyperplane arrangement} $\A$ is 
a finite collection of hyperplanes in $V$. An arrangement $\A$ is called \emph{central} if every $H \in \A$ is linear, and \emph{essential} if $\bigcap_{H \in \A} H$ is the origin. 
For each $H \in \A$ we fix a linear form $\alpha_H \in V^*$ such that $\ker \alpha_H=H$ and define $Q(\A)\coloneqq\prod_{H \in \A} \alpha_H$. In this article 
we assume that all arrangements are essential unless otherwise specified. 

Let $m:\CA\to\mathbb{Z}_{\geq 0}$ be a multiplicity on $\CA$. Define the order of $m$ as $|m|\coloneqq\sum_{H \in \A}m(H)$ and define $Q(\A,m)\coloneqq\prod_{H \in \A} \alpha_H^{m(H)}$. 

If $m = \bbmo$, we call $(\A,m)$ a \emph{simple arrangement}. 
For a central multiarrangement $(\A,m)$, we can define the \emph{module of logarithmic $(\A,m)$-derivations} $D(\A,m)$ by 
\begin{equation}
    \label{def:DAm}
    D(\A,m)\coloneqq\{\theta \in \Der_S\mid 
\theta(\alpha_H) \in S\alpha_H^{m(H)}\ (\forall H \in \A)\}.
\end{equation}
Then $D(\A,m)$ is an $S$-graded reflexive module of rank $\ell$. Define $D(\A,m)_k=\{\theta\in D(\A,m)\mid \deg\theta=k\}$ and $D(\A,m)_{<k}=\{\theta\in D(\A,m) \mid \deg\theta<k\}$.
When $D(\A,m)$ is a free module of rank $\ell$, we say that 
$(\A,m)$ is \emph{free} and for a homogeneous basis $\theta_1,\ldots,
\theta_\ell$, we define the \emph{exponents of $(\CA,m)$} by
$$
\exp(\A,m)\coloneqq(\deg \theta_1,\ldots,\deg \theta_\ell),
$$
where $\deg \theta\coloneqq\deg \theta(\alpha)$ for some $\alpha \in V^*$ such that 
$\theta(\alpha)\neq 0$. 
When $m = \bbmo$, 
the module $D(\A,\bbmo)$ is also denoted by $D(\A)$. 

The notion of criticality was  originally defined by Ziegler in \cite{Z2} for the $S$-module of logarithmic differential forms $\Omega^1(\A)$, which is isomorphic to the dual $D(\A)^*$ of the logarithmic derivation module.

\begin{defn}\label{definition: k-critical}
We say that $(\A,m)$ is \emph{$k$-critical} if $D(\A,m)_{k}\neq (0),\ 
D(\A,m)_{<k}=(0)$ and for any $H \in \A$, $D(\A,m+\delta_H)_{k}=(0).$
\end{defn}

The following results are fundamental in this article, see \cite{Sa} and \cite{Z}.

\begin{theorem}[Saito's criterion]\label{Saito}
Let $\theta_1,\ldots,\theta_\ell$ be homogeneous derivations in $D(\A,m)$. 
Then $\A$ is free with basis $\{\theta_1,\ldots,\theta_\ell\}$ if and only if $\{
\theta_1,\ldots,\theta_\ell\}$ is $S$-independent and $\sum_{i=1}^\ell \deg \theta_i=|m|$.

Equivalently, $\A$ is free with basis $\{\theta_1,\ldots,\theta_\ell\}$ if and only if
$\det (\theta_i(x_j)) = c\cdot Q(\A,m)$, where $c \in \BBK\setminus \{0\}$.
\end{theorem}

\begin{defn}
	\label{def:Euler}
	Let $(\CA, \mu),$ (where $\CA\neq\emptyset)$ be a multiarrangement in $\mathbb{K}^\ell$. Fix $H_0$ in $\CA$.
	We define the \emph{deletion}  $(\CA', \mu')$ and \emph{Euler restriction} $(\CA'', \mu^*)$
	of $(\CA, \mu)$ with respect to $H_0$ as follows.
	If $\mu(H_0) = 1$, then set $\CA' = \CA \setminus \{H_0\}$
	and define $\mu'(H) = \mu(H)$ for all $H \in \CA'$.
	If $\mu(H_0) > 1$, then set $\CA' = \CA$
	and define $\mu'(H_0) = \mu(H_0)-1$ and
	$\mu'(H) = \mu(H)$ for all $H \ne H_0$.
	
	Let $\CA'' = \{ H \cap H_0 \mid H \in \CA \setminus \{H_0\}\ \}$.
	The \emph{Euler multiplicity} $\mu^*$ of $\CA''$ is defined as follows.
	Let $Y \in \CA''$. Since the localization $\CA_Y$ is of rank $2$, the
	multiarrangement $(\CA_Y, \mu_Y)$ is free, 
	\cite[Cor.~7]{Z}. 
	According to 
	\cite[Prop.~2.1]{ATW2},
	the module of derivations 
	$D(\CA_Y, \mu_Y)$ admits a particular homogeneous basis
	$\{\theta_Y, \psi_Y, \partial_{x_3}, \ldots, \partial_{x_\ell}\}$,
	such that $\theta_Y \notin \alpha_0 \Der(S)$
	and $\psi_Y \in \alpha_0 \Der(S)$,
	where $H_0 = \ker \alpha_0$.
	Then on $Y$ the Euler multiplicity $\mu^*$ is defined
	to be $\mu^*(Y) = \deg \theta_Y$.
\end{defn}

\begin{defn}\label{definition: balanced}
    Let $(\A,m)$ be a multiarrangement. We call the multiplicity $m$ \emph{balanced} if for all $H\in\CA$ the inequality $$m(H)\leq \sum_{H'\in \CA\backslash\{H\}}m(H')$$ holds.
\end{defn}

\begin{lemma}
\label{lemma: nabla and theta}
    Let $\theta,\theta'\in\Der_S$ and $\alpha\in V^*$, then $$(\nabla_\theta \theta')(\alpha)=\theta(\theta'(\alpha)).$$
\end{lemma}
\begin{proof}
    Let $\theta=\sum_{i=1}^\ell f_i\partial_{x_i}, \theta'=\sum_{i=1}^\ell g_i\partial_{x_i}$, and $\alpha=\sum_{i=1}^\ell c_ix_i$. Then we have
   $$ \left(\nabla_\theta \theta'\right)(\alpha) = \sum_{i=1}^\ell \theta(g_i)\partial_{x_i}(\alpha)                                   = \sum_{i=1}^\ell c_i\theta(g_i) = \theta\left(\sum_{i=1}^\ell c_ig_i\right)=\theta(\theta'(\alpha)), $$
as claimed.
\end{proof}

\begin{proposition}\label{indep}
Let $\theta \in \Der_S$. Then $\theta$ is $m$-universal for $\A$ if and only if $
\{\nabla_{\partial_{x_i}} \theta\}_{i=1}^\ell$ is independent over $S$, 
$\ell\cdot(\deg \theta-1)=|m|$, and $\nabla_{\partial_{x_i}} \theta \in D(\A,m)$ for all $i$. Moreover, $D(\A,\mu)\simeq D(\A,m+\mu)$ for any multiplicy $\mu:\A\to\{0,1\}$.
\end{proposition}
\begin{proof}
For the equivalence it is sufficient to show that $\theta$ belongs to $D(\A,m+\bbmo)$. Let $\theta=\sum_{i=1}^\ell g_i \partial_{x_i} \in \Der_S$ and $H\in\A$. We may assume, without loss of generality, that $\partial_{x_1} 
(\alpha_H)=c \neq 0$ and that $\theta(\alpha_H)=h\alpha_H^{n}$ with $
\alpha_H \nmid h\in S, n\geq 0$. Since $\nabla_{\partial_{x_1}} \theta \in D(\A,m)$, it follows from Lemma \ref{lemma: nabla and theta}  that 
$$
S\alpha_H^{m(H)} \ni (\nabla_{\partial_{x_1}}\theta)(\alpha_H)=\partial_{x_1}(\theta(\alpha_H))=\partial_{x_1}(h\alpha_H^n)
=(\alpha_H\partial_{x_1}h+n hc)\alpha_H^{n-1}.
$$
So $n \ge m(H)+1$, showing that $\theta \in D(\A,m+\bbmo)$.  Assume that $n >m(H)+1$. Then it holds for all $1\leq i\leq \ell$ that $$
\nabla_{\partial_{x_i}} \theta\in D(\A,m+\bbmo+\delta_H),
$$
and these derivations are linearly independent. But this contradicts Saito's criterion and therefore we have $n =m(H)+1$. The only thing left to prove is the isomorphism claim. So fix a multiplicity $\mu:\A \rightarrow \{0,1\}$. We show that the map 
$$
\Phi:D(\A,\mu) \rightarrow D(\A,m+\mu)
$$
defined by $\Phi(\varphi)\coloneqq\nabla_\varphi \theta$ is an isomorphism of $S$-modules. Let $\varphi=\sum_{i=1}^\ell f_i \partial_{x_i} \in D(\A,\mu)$ and $\nabla_\varphi \theta=0$. Then $\nabla_{\partial_{x_i}}\theta=\sum_{j=1}^\ell \partial_{x_i}(g_j)\partial_{x_j}$ gives   

$$\nabla_\varphi \theta=\sum_{j=1}^\ell \varphi(g_j) \partial_{x_j}=\sum_{i=1}^\ell f_i\left(\sum_{j=1}^\ell \partial_{x_i}(g_j) \partial_{x_j}\right)=\sum_{i=1}^\ell f_i \nabla_{\partial_{x_i}}\theta=0.$$ 

Since the elements of $\{\nabla_{\partial_{x_i}} \theta\}_{i=1}^\ell$ are independent over $S$, it follows that $f_i=0$ for all $i$. Thus $\Phi$ is injective.

To show that $\Phi$ is surjective, let $\eta \in D(\A,m+\mu)$ be arbitrary. We construct a $\varphi \in D(\A,\mu)$ such that $\nabla_\varphi \theta=\eta$. Since $\theta$ is $m$-universal, the map $\Phi_\theta: D(\A,0)=
\Der_S \rightarrow D(\A,m)$ is an isomorphism and $\nabla_{\partial_{x_1}} \theta,\ldots,\nabla_{\partial_{x_\ell}} \theta$ form an $S$-basis for $D(\A,m)$. Consequently, since $\eta \in D(\A,m+\mu) \subset D(\A,m)$, there exist unique $f_i \in S$ such that $\eta=\sum_{i=1}^\ell f_i \nabla_{\partial_{x_i}} \theta$. Therefore, defining $\varphi\coloneqq\sum_{i=1}^\ell f_i \partial_{x_i} \in \Der_S$, we derive
$$
\eta=\sum_{i=1}^\ell f_i \nabla_{\partial_{x_i}} \theta 
=\sum_{i=1}^\ell f_i\left(\sum_{j=1}^\ell \partial_{x_i}(g_j) \partial_{x_j}\right)=\sum_{j=1}^\ell \varphi(g_j) \partial_{x_j}=\nabla_{\varphi }\theta.
$$
So it suffices to show that $\varphi \in D(\A,\mu)$, since $\nabla_\varphi \theta =\eta \in D(\A,m+\mu)$. Let $H \in \A$. There is nothing to show if $\mu(H)=0$, so assume that $\mu(H)=1$. We have to show that $\varphi(\alpha_H) \in S\alpha_H$. Use $\eta\in D(\A,m+\delta_H)$ to obtain
\begin{align*}
S\alpha_H^{m(H)+1} \ni \eta(\alpha_H) & =\nabla_\varphi \theta(\alpha_H)
 =\varphi(\theta(\alpha_H))=
\varphi(h\alpha_H^{m(H)+1})\\
&   =
(\varphi(h) \alpha_H+(m(H)+1)\varphi(\alpha_H)h)\alpha_H^{m(H)}. 
\end{align*}

Since $\alpha_H \nmid h$, we have $\varphi(\alpha_H) \in S\alpha_H$. So $\varphi \in D(\A,\mu)$, which completes the proof.\end{proof}

\begin{corollary}\label{corollary: if universal then critical}
    Let $\A$ be an irreducible arrangement and $m$ a multiplicity on $\A$, and 
    let $\theta\in D(\A,m+\bbmo)$ be $m$-universal. Then $\theta$ is (up to scalar multiplication) the unique derivation of degree $\deg\theta$ in $D(\A,m+\bbmo)$ and $D(\A,m+\bbmo)_{<\deg\theta}=(0)$.
\end{corollary}
\begin{proof}
    Since $\A$ is irreducible, the Euler derivation $\theta_E$ is (up to scalar multiplication) the unique derivation of degree $1$ in $D(\A)$. The map $\Phi: D(\A)\to D(\A,m+\bbmo)$ given by  $\Phi(\varphi)=\nabla_{\varphi}\theta$  
     is an isomorphism, by the proof of Proposition \ref{indep}. Since $\deg(\nabla_{\varphi}\theta)=\deg \varphi + \deg \theta - 1$, for $\varphi$ homogeneous, we require $\Phi(\theta_E)=\theta$ which finishes the proof.
\end{proof}

For the remainder of this section, let $\A$ be a $2$-arrangement.
We require some results from \cite{AN}. Given an arrangement $\A$ of rank two and a multiplicity $m:\A \rightarrow \Z_{\ge 0}$, we let $\exp (\A,m)=(d_1,d_2)$. Then we define 
$$
\Delta(m)\coloneqq|d_1 - d_2|.
$$

\begin{theorem}[{\cite[Thm.~0.1]{A5}}]\label{limit of rank 2 exponents}
Let $(\A,m)$ be a balanced $2$-multiarrangement with exponents
$\exp(\A,m)=(d_1,d_2)$. Then 
$\Delta(m) \le |\A|-2$. 
\end{theorem}

Now consider the \emph{multiplicity lattice} $\Lambda$ of $\A$ and its subset $\Lambda'$:
\begin{eqnarray*}
\Lambda&\coloneqq&\{m:\A \rightarrow \Z_{\ge 0}\},\\
\Lambda'&\coloneqq&\{m \in \Lambda\mid \Delta(m) \neq 0\}.
\end{eqnarray*}
The partial order on $\Lambda$ is given by $m\leq m' \iff m(H)\leq m'(H)$ for all $H\in\A$. We view $\Delta$ as a map on  $\Lambda$:
$$
\Delta:\Lambda \rightarrow \Z_{\ge 0}.
$$
In addition, 
for a fixed hyperplane $H \in \A$ define the corresponding set of \emph{unbalanced multiplicities} $\Lambda_H$ and the set of \emph{balanced multiplicities $\Lambda_0$ on $\A$}  by
\begin{eqnarray*}
\Lambda_H&\coloneqq&\left\{ m \in \Lambda \ \middle| \ 
m(H) > \sum_{H' \in \A \setminus \{H\}} m(H')\right\}\,\\
\Lambda_0&\coloneqq&\Lambda' \setminus \bigcup_{H \in \A} \Lambda_H.
\end{eqnarray*}
Note, $\Lambda_0$ is denoted by $\Lambda'_{\phi}$ in \cite{AN}. 
Define the \emph{distance $d(m,m')$ between $m,m' \in \Lambda$} by 
$$
d(m,m')\coloneqq\sum_{H \in \A} | m(H)-m'(H)|. 
$$
We say that two points $m,m' \in \Lambda$ are \emph{connected} 
if there exist points $m_1,\ldots,m_n \in \Lambda$ such that 
$d(m_i,m_{i+1})=1\ (i=1,\ldots,n-1)$ and $m_1=m,m_n=m'$. 
This allows us to define \emph{connected components} of $\Lambda$. 
Likewise, we define connected components of $\Lambda_0$. 

\begin{lemma}[{\cite[Lem.~4.2]{AN}}]\label{one}
For $m_1,m_2 \in \Lambda$ with $m_1\leq m_2$ and
$d(m_1,m_2)=1$, we have 
$|\Delta(m_1)-\Delta(m_2)|=1$.
\end{lemma}

The notion of $\emph{peak points}$ of $\A$ is introduced in the following result.

\begin{theorem}[{\cite[Thm.~3.2]{AN}}]\label{str}
Let $C \subset \Lambda_0$ be a connected component of $\Lambda_0$. 
Then there exists a unique point $m \in C$, called the \emph{peak point of $C$}, such that 
$$
\Delta(m) \ge \Delta(m')\ \forall m' \in C.
$$
Moreover, 
$$
C=\{m' \in \Lambda' \mid d(m,m') < \Delta(m)\}
$$
and for $m' \in C$, 
$$
\Delta(m')=\Delta(m)-d(m,m').
$$
\end{theorem}
We call a connected component of $\Lambda_0$ in Theorem \ref{str} a \emph{finite component}  of $\Lambda$, and every $\Lambda_H$ an \emph{infinite component} of $\Lambda$. 

\begin{theorem}[{\cite[Lem.~4.17]{AN}}]\label{ANindependent}
Let $C_1\neq C_2$ be two connected components in $\Lambda_0$, such that there exist $m_i \in C_i$ with $d(m_1,m_2)=2$. 
Let $\theta_i \in D(\A,m_i)$ be homogeneous basis elements of lower degree for $i = 1,2$. Then $\theta_1$ and $\theta_2$ are linearly independent over $S$.
\end{theorem}

\section{Proofs of Theorems \ref{theorem: alternative universal criterion} and \ref{theorem: rank two universal criterion}}
\label{sect: main result proof}

\begin{lemma}\label{lemma20}  
Let $\theta$ be $m$-universal for $(\A,m)$. Then $\theta \in D(\A,m+\bbmo) \setminus 
D(\A,m+\bbmo+\delta_H)$ for every $H \in \A$. 
Indeed, $(\A,m+\bbmo)$ is $(d+1)$-critical, where  $d+1\coloneqq\deg \theta$.
\end{lemma}
\begin{proof}
Let $\theta\in D(\A,m+\bbmo)$ be $m$-universal. Then by definition $\{\nabla_{\partial_{x_i}}\theta\}_{i=1}^\ell$ forms a basis for $D(\A,m)$. 
Then for $M=(\nabla_{\partial_{x_i}}\theta(x_j))_{1\leq i,j\leq \ell}$, it follow from Theorem \ref{Saito} that 
$$0\neq \det M=c\cdot Q(\A,m),$$
where $c \in \BBK\setminus \{0\}$. Suppose $D(\A,m+\bbmo)$ is not $(d+1)$-critical. Then by Corollary \ref{corollary: if universal then critical} there is a hyperplane $H\in\CA$ such that $\theta\in D(\A,m+\bbmo+\delta_H)$. By Lemma \ref{lemma: nabla and theta}, we have $$(\nabla_{\partial_{x_i}}\theta)(\alpha_H)=\partial_{x_i}(\theta(\alpha_H))$$ and therefore  $\nabla_{\partial_{x_i}} \theta\in D(\A,m+\delta_H)$ for all $1\leq i\leq \ell$. Finally, we derive 
$$0\neq \det M \in S\cdot Q(\A,m+\delta_H)=S\cdot (Q(\A,m)\cdot\alpha_H),$$
a contradiction.
\end{proof}

\begin{proof}[Proof of Theorem \ref{theorem: alternative universal criterion}] 
First assume (1). 
Since $\theta$ is $m$-universal and $\deg \theta=d+1$, the set $\{\nabla_{\partial_{x_i}}\theta \mid 1 \le i \le \ell\}$ forms a basis for $D(\A,m)$ and $\deg \nabla_{\partial_{x_i}}\theta = d$ for each $1\leq i \leq \ell$. Lemma \ref{lemma20} shows that $(\A,m+\bbmo)$ is $(d+1)$-critical, so we have (2). 

Now suppose (2). Without loss, we may assume that $\{\ker x_i \mid 1\leq i\leq\ell\} \subset \A$. Let $m_i\coloneqq (m+\bbmo)(\ker x_i)$. 
Then for $\theta(x_i)\coloneqq f_{i}x_i^{n_{i}}$ with $n_{i} \ge m_i$, we have $f_{i} \neq 0$ and $x_i \nmid f_i$, by the criticality assumption. It follows that 
$$
(\nabla_{\partial_{x_i}} \theta)(x_i)
=\partial_{x_i}(\theta(x_i))
=\partial_{x_i}(f_{i}x_i^{n_{i}})
=\partial_{x_i}(f_{i})x_i^{n_{i}}+f_{i}n_i x_i^{n_{i}-1}
\neq 0. 
$$
Moreover, since $\theta(\alpha_H)\in S\cdot\alpha_H^{(m+\bbmo)(H)}$ also holds for any other $H\in \A$, it follows that 
$$
(\nabla_{\partial_{x_i}} \theta)(\alpha_H)
=\partial_{x_i}(\theta(\alpha_H))
\in S\cdot\alpha_H^{m(H)}
$$
and thus $D(\A,m) \ni \nabla_{\partial_{x_i}} \theta \neq 0$ for $i=1,\ldots,\ell$. 

Since $\deg \nabla_{\partial_{x_i}} \theta=d$ for all $i$ 
and $(\A,m)$ is free with $\exp(\A,m)=(d,\ldots,d)$, 
it suffices to show that $\{\nabla_{\partial_{x_i}}\theta \mid 1 \le i \le \ell\}$
 is independent over $\K$. 
Assume that 
$$
\sum_{i=1}^\ell c_i\nabla_{\partial_{x_i}} \theta=0
$$
for some $c_i \in \K$. 
Then for all $j=1,\ldots,\ell$ 
$$
\sum_{i=1}^\ell c_i (\nabla_{\partial_{x_i}} \theta)(x_j)=0.
$$
This is equivalent to 
\begin{align*}
    0 = \sum_{i=1}^\ell c_i \partial_{x_i}(f_{j}x_j^{n_{j}})
    &=c_j \partial_{x_j}(f_{j}x_j^{n_{j}})+
    \sum_{i\neq j}^\ell c_i \partial_{x_i}(f_{j}x_j^{n_{j}})\\
    &=c_j (\partial_{x_j}(f_{j})x_j^{n_{j}}+f_{j}\partial_{x_j}(x_j^{n_{j}}))+
    \sum_{i\neq j}^\ell c_i 
    (\partial_{x_i}(f_{j})x_j^{n_{j}}+f_{j}\partial_{x_i}(x_j^{n_{j}}))\\
    &=c_j (\partial_{x_j}(f_{j})x_j^{n_{j}}+f_{j}n_jx_j^{n_{j}-1})+
    \sum_{i\neq j}^\ell c_i 
    \partial_{x_i}(f_{j})x_j^{n_{j}}\\
    &=c_j f_{j}n_jx_j^{n_{j}-1}+
    \sum_{i=1}^\ell c_i 
    \partial_{x_i}(f_{j})x_j^{n_{j}}.
\end{align*}

Since $x_j \nmid f_j$, 
this equality shows that $c_j=0$, 
showing that $\{\nabla_{\partial_{x_i}}\theta \mid 1 \le i \le \ell\}$ 
is independent over $\K$, which completes the proof.
\end{proof}

Next, we focus on arrangements of rank two. We require the following lemma.

\begin{lemma}\label{balanced}
Let $(\A,m)$ be a central $2$-multiarrangement such that 
$\Delta(m) \neq 0$ and let $\theta \in D(\A,m)$ be a basis element of lower 
degree. Then 
$(\A,m)$ is balanced if and only if $\theta(\alpha_H) \neq 0$ for 
every $H \in \A$.
\end{lemma}
\begin{proof}
Suppose that  $(\A,m)$ is not balanced. Then there exists a hyperplane 
$H \in \A$ such that $m(H) > \sum_{H_1 \in \A\setminus\{H\}} m(H_1)$. 
We may assume that $\alpha_H=x_1$. Then it is easy to see that 
$\theta=\prod_{H_1 \in \A\setminus\{H\}} \alpha_{H_1} \partial_{x_2}$. 
Hence $\theta(x_1)=0$. Conversely, assume that 
$\theta(x_1)=0$. Now let $\varphi \in D(\A,m)$ be 
a higher degree basis element. Then 
$\{\theta,(\alpha_H)^{s} \varphi\}$ forms a basis 
for $D(\A,m+s\delta_H)$. Hence $m \in \Lambda$ belongs to 
an infinite component and is not balanced, which completes the proof. 
\end{proof}

To prove Theorem \ref{theorem: rank two universal criterion}, we first utilize Lemma \ref{balanced} to show that a  basis element of lower degree is an $m$-universal derivation. 

\begin{theorem}\label{univ}
Let $\A$ be an arrangement in $\K^2$ with $|\A| >2$ and 
$m+\bbmo:\A \rightarrow \Z_{> 0}$ 
a balanced multiplicity such that 
$\Delta(m+\bbmo)= |\A|-2$.
Assume that one of the following holds:
\begin{itemize}
\item[(1)]
$|\A|=3$ and $m$ is balanced, or 
\item[(2)]
$|\A| \ge 4$.
\end{itemize}
Then a lower degree basis element $\theta_0$ for $D(\A,m+\bbmo)$ is $m$-universal.
\end{theorem}
\begin{proof} 
The assumptions and Theorem \ref{limit of rank 2 exponents}  
imply that $m+\bbmo$ is the peak point of 
some finite component $C \subset \Lambda_0$. 
We may assume that $\{\ker x_1, \ker x_2\} \subset \A$ and take 
$\partial_{x_1}, \partial_{x_2}$ such that 
$\partial_{x_i}(x_j)=\delta_{ij}$ for $\{i,j\}=\{1,2\}$. 
By Proposition \ref{indep}, it suffices to show that
$\nabla_{\partial_{x_i}} \theta_0 \in D(\A,m)$ for $i=1,2$, and 
$\nabla_{\partial_{x_1}} \theta_0$ and $\nabla_{\partial_{x_2}} \theta_0$ are 
$S$-independent. Let $\delta_i := \delta_{\ker x_i}$  be the indicator multiplicity of $\ker x_i$ for  $i\in\{1,2\}$. 
Then $\nabla_{\partial_{x_i}} \theta_0 \in 
D(\A,m+\delta_{3-i})$ 
by the same arguments as in the proof of Theorem \ref{theorem: alternative universal criterion}. 
Since $d(m+\bbmo,m+\delta_i)=|\A|-1$, 
it follows from Lemma \ref{one} and 
Theorem \ref{str} that $\exp(\A,m+\delta_i)=
(\deg \theta_0-1, \deg \theta_0)$. Since  $\nabla_{\partial_{x_i}} \theta_0 \neq 0$ by Lemma 
\ref{balanced},  
$\nabla_{\partial_{x_i}} \theta_0$ is of degree $\deg \theta_0-1$ and is a lower degree basis element 
for $D(\A,m+\delta_{3-i})$ for $i=1,2$. 

Note that 
$\Delta(m+\delta_i)=1$ and $\Delta(m+\delta_1+\delta_2)=0$ by Theorem \ref{str}. 
Hence, by Theorem \ref{theorem: alternative universal criterion} or by Theorem \ref{ANindependent}, it suffices to show that 
$\Delta(m)=0$. 
Suppose that $\Delta(m) \neq 0$. Then Lemma \ref{one} shows that 
$\Delta(m)=2$. 
Let $\theta'$ be a lower degree basis element for $D(\A,m)$. 
Since $x_i \theta'$ is a lower degree basis element 
for $D(\A,m+\delta_i)$ for $i=1,2$, and 
$\Delta(m+\delta_i)=1$, it follows that 
$$
\nabla_{\partial_{x_1}} \theta_0=x_2  \theta', \ 
\nabla_{\partial_{x_2}} \theta_0=x_1  \theta'
$$
up to non-zero scalars.
Since 
$$
\theta_0=
\nabla_{\theta_E} \theta_0=
x_1\nabla_{\partial_{x_1}} \theta_0+
x_2\nabla_{\partial_{x_2}} \theta_0
=2x_1x_2 \theta'
$$
up to non-zero scalars, 
$\theta_0$ and $\theta'$ are $S$-dependent, contradicting 
Theorem \ref{ANindependent}. 
Hence $\Delta(m)=0$ and 
Proposition \ref{indep} completes the proof. 
\end{proof}

\begin{proof}[Proof of Theorem \ref{theorem: rank two universal criterion}] 
Owing to  Theorem \ref{univ} the reverse implication holds. For the forward implication take an $m$-universal derivation $\theta \in D(\A,m+\bbmo)$ with $\deg \theta = d+1$.  
Since $|\A| \ge 3$, Theorem \ref{univ} implies 
that $\theta$ is a lower degree basis element for $D(\A,m+\bbmo)$. 

First suppose that $(\A,m+\bbmo)$ is not balanced. 
We may assume that $\ker x \in \A$ with $(m+\bbmo)(\ker x)=m_0$ satisfying $2m_0 > |m+\bbmo|$. 
Then $\exp(\A,m+\bbmo)=(|m+\bbmo|-m_0,m_0)$ with $m_0 > |m+\bbmo|-m_0$. So $|m+\bbmo|-m_0=d+1$ since the $m$-universal derivation must be a  lower degree basis element of $D(\A,m+\bbmo)$. It is easy to see that 
$$
\theta_1\coloneqq \displaystyle \frac{Q(\A,m+\bbmo)}{x^{m_0}} \displaystyle \frac{\partial}{\partial y}
$$
is a lower degree basis element for $D(\A,m+\bbmo)$, so it has to be the $m$-universal $\theta$ (up to a non-zero scalar factor). However,
$\{\nabla_{\partial_x} \theta, \nabla_{\partial_y} \theta\}$
is $S$-dependent, which is a contradiction. So $(\A,m+\bbmo)$ is balanced.

If $\theta_E,\varphi$ form a homogeneous basis for $D(\A)$, then apply Proposition \ref{indep} and Corollary \ref{corollary: if universal then critical} to see that $\nabla_{\theta_E} \theta = \theta$ and $\theta_2\coloneqq \nabla_\varphi \theta$ form a homogeneous basis for $D(\A,m+\bbmo)$. So 
$$
\deg \theta_2-\deg \theta=\deg \varphi-\deg \theta_E=|\A|-2,
$$
which completes the proof.
\end{proof}
 
In Example \ref{ex:a2}, we demonstrate how Theorem \ref{theorem: rank two universal criterion} can be utilized. We require the following result due to Wakamiko.

\begin{theorem}[{\cite[Thm.~1.5]{W07}}]\label{theorem: Wakamiko A2 classification}
	Let $\CA=\{H_1,H_2,H_3\}$ be a $2$-arrangement of three lines, $m$ a multiplicity on $\CA$ with $m(H_i)=k_i,(i=1,2,3)$ and $\exp(\CA,m)=(d_1,d_2)$. Assume that $k_3\geq\max\{k_1,k_2\}$ and let $k=k_1+k_2+k_3$.
	\begin{enumerate}
		\item If $k_3< k_1+k_2-1$, then  
		\begin{equation*}
			\vert d_1-d_2\vert =
			\begin{cases}
				0 & \text{if $k$ is even,}\\
				1 & \text{if $k$ is odd.}\\
			\end{cases}       
		\end{equation*}
		\item If $k_3 \geq k_1+k_2-1$, then $\exp(\CA,m)=(k_1+k_2,k_3)$.
	\end{enumerate}
\end{theorem}

\begin{example}\label{ex:a2}
Consider the Coxeter multiarrangement $(\A,m)$ of type $A_2$ defined by 
$$
Q(\A,m)=x^a y^b (x-y)^c.
$$
The multiarrangement $(\A,m)$ is always free and its exponents are given by Theorem \ref{theorem: Wakamiko A2 classification}. Now Theorem \ref{theorem: rank two universal criterion} shows that a lower degree element of $D(\A,m+\bbmo)$ is $m$-universal if and only if both $m$ and $m+\bbmo$ are balanced and $\vert m\vert$ is even.
\end{example}

\section{Examples}\label{sect: examples}
In this section we demonstrate how Theorems \ref{theorem: alternative universal criterion} and \ref{theorem: rank two universal criterion} can be used to derive the existence of universal derivations. Since the majority of our examples concern supersolvable arrangements, we begin by recalling their definition and derive some first corollaries.

\subsection{Supersolvable multiarrangements, I}
The following result by Abe, Terao, and Wakefield plays a key role in the sequel.

\begin{theorem}[{\cite[Thm.~5.10]{ATW2}}]\label{theorem: supersolvable exponents}
    Let $(\A,m)$ be a multiarrangement such that $\A$ has a supersolvable filtration $\A_1\subset\A_2\dots\subset\A_r=\A$ and $r\geq 2$. Let $m_i$ denote the multiplicity $m\vert_{\A_i}$ on $\A_i$ and $\exp(\A_2,m_2)=(d_1,d_2,0,\dots,0)$. Assume that for each $H'\in \A_d\backslash \A_{d-1}, H''\in \A_{d-1} (d=3,\dots, r)$ and $X\coloneqq H'\cap H''$, either that $$\A_X=\{H',H''\}$$
    or
    $$m(H'')\geq \left(\sum_{X \subset H\in(\A_d\backslash \A_{d-1})} m(H)\right)-1.$$
    Then $(\A,m)$ is free with $\exp(\A,m)=(d_1,d_2,\vert m_3\vert-\vert m_2\vert,\dots, \vert m_r\vert-\vert m_{r-1}\vert,0,\dots,0)$.
\end{theorem}

\begin{defn}\label{definition: supersolvable multi}
    Let $(\A,m)$ be a multiarrangement. If there exists a supersolvable filtration for $\A$ such that the conditions in Theorem \ref{theorem: supersolvable exponents} are satisfied, then we call $(\A,m)$ \emph{supersolvable}. 
\end{defn}

The following observation gives a necessary condition for the presence of a universal derivation in the supersolvable case.

\begin{lemma}\label{lemma: breaking supersolvability for Multi-Euler}
    Let $(\CA,m+\bbmo)$ be supersolvable in $V=\mathbb{K}^\ell$ with a supersolvable filtration $\CA_1\subset\CA_2\subset\dots\subset\CA_r=\CA$ as in Theorem \ref{theorem: supersolvable exponents}. 
    If there exists an $m$-universal derivation in $D(\A,m+\bbmo)$, then for any hyperplane $H\in\CA\backslash\CA_2$, this filtration is not  supersolvable for $(\CA,m+\bbmo+\delta_H)$.
\end{lemma}
\begin{proof}
     Let $(\CA,m+\bbmo)$ be supersolvable, $\rank(\CA)=r\geq3$, $$\exp(\CA,m+\bbmo)=(d_1,d_2,d_3,\dots,d_r,0,0,\dots,0),$$ and choose a supersolvable filtration $$(\CA_1,m_1+\bbmo) \subset (\CA_2,m_2+\bbmo)\subset\dots \subset (\CA_r,m_r+\bbmo) = (\CA,m+\bbmo)$$ for $(\CA,m+\bbmo)$ as in Theorem \ref{theorem: supersolvable exponents}. Let $\exp(\CA_2,m_2+\bbmo)=(d_1,d_2)$ and assume that there exists an $m$-universal derivation $\theta \in D(\A,m+\bbmo)$. Note that for all $H\in\CA_2$ increasing the multiplicity of $H$ by $1$ preserves the supersolvability of $(\CA,m+\bbmo)$. Therefore,  by Theorem \ref{theorem: supersolvable exponents}, the multiarrangement $(\CA,m+\bbmo+\delta_H)$ is free with $$\exp(\CA,m+\bbmo+\delta_H)=(d_1+1,d_2,d_3,\dots,d_r,0,\dots,0),$$ for all $H\in\CA_2$. Since $(\A,m+\bbmo)$ admits a universal derivation, it follows from Theorem \ref{theorem: alternative universal criterion} that $(\A,m+\bbmo)$ is $d_1$-critical and $d_1<d_i$ for $2\leq i\leq r$. Now let $H\in\CA\backslash\CA_2$ be arbitrary. If $(\CA,m+\bbmo+\delta_H)$ is still supersolvable, then again by Theorem \ref{theorem: supersolvable exponents} there exists a $3\leq k\leq r$ such that $$\exp(\CA,m+\bbmo+\delta_H)=(d_1,d_2,\dots,d_{k-1},d_k+1,d_{k+1},\dots,d_r,0,\dots,0).$$ Since $d_1<d_k$, this shows that $D(\CA,m+\bbmo+\delta_H)_{d_1}\neq (0)$ which contradicts the criticality. 
     So $(\CA,m+\bbmo+\delta_H)$ is not supersolvable for the chosen filtration.
\end{proof}

We list further restrictions on the cardinality and parity for multiplicities to admit a universal derivation.

\begin{lemma}
   Let $(\CA,m+\bbmo)$ be supersolvable in $V=\mathbb{K}^\ell$ with a supersolvable filtration $\CA_1\subset\CA_2\subset\dots\subset\CA_r=\CA$ and notation as in Theorem \ref{theorem: supersolvable exponents}. Suppose $(\CA,m+\bbmo)$ admits an $m$-universal derivation. Then:
    \begin{enumerate}
        \item The multiplicity $m_2+\bbmo$ on $\CA_2$ is balanced.
        \item $\vert m_2\vert$ is even and $\exp(\CA_2,m_2)=\left(\frac{\vert m_2\vert}{2},\frac{\vert m_2\vert}{2}\right)$. 
        \item $\vert m_i\vert-\vert m_{i-1}\vert=\frac{\vert m_2\vert}{2}$ for all $i=3,\dots,r$.
        \end{enumerate}
\end{lemma}
\begin{proof}
    1. Suppose that $(\CA_2,m_2+\bbmo)$ is not balanced and let $H\in\CA_2$ such that $(m_2+\bbmo)(H)> \sum_{H'\in \CA_2\backslash\{H\}}(m_2+\bbmo)(H')$. Then $$\exp(\CA_2,m_2+\bbmo)=\left(\sum_{H'\in \CA_2\backslash\{H\}}(m_2+\bbmo)(H'),(m_2+\bbmo)(H)\right)$$
    and 
    $$\exp(\CA_2,m_2)=\left(\sum_{H'\in \CA_2\backslash\{H\}}m_2(H'),m_2(H)\right),$$ since $(\CA_2,m_2)$ is still not balanced. In particular, we cannot have equal exponents for $(\CA_2,m_2)$.\\
    2. For any free multiarrangement $(\CA,m)$, the sum of its exponents $\exp(\CA,m)$ equals $\vert m\vert$. So $\vert m_2\vert$ has to be even as the exponents of $(\CA,m_2)$ need to be equal.\\
    3. If $(\CA,m+\bbmo)$ is supersolvable, then $(\CA,m)$ is still supersolvable. So $$\exp(\CA,m)=(\exp(\CA_2,m_2),\vert m_3\vert-\vert m_2\vert,\dots,\vert m_r\vert-\vert m_{r-1}\vert).$$ For $D(\CA,m+\bbmo)$ to have an $m$-universal derivation we require that all exponents are equal. So $\exp(\CA_2,m_2)=(\frac{\vert m_2\vert}{2},\frac{\vert m_2\vert}{2})$ and $$\frac{\vert m_2\vert}{2}=\vert m_3\vert-\vert m_2\vert=\dots=\vert m_r\vert-\vert m_{r-1}\vert$$ has to hold. Therefore, $\vert m_i\vert-\vert m_{i-1}\vert=\frac{\vert m_2\vert}{2}$ for all $i=3,\dots,r$.
\end{proof}

\subsection{Deleted $A_3$}
Now we demonstrate how Theorems \ref{theorem: alternative universal criterion} and \ref{theorem: supersolvable exponents} can be utilized to show the existence of universal derivations. We start by investigating the deleted $A_3$ arrangement. All free multiplicities on the deleted $A_3$ arrangement were classified by Abe.

\begin{theorem}[{\cite[Thm.~0.2]{A07}}]\label{theorem: del A3 classification}
    Let $\A$ be the deleted $A_3$ arrangement given by $Q(\A) = xy(x-y)(x-z)(y-z)$. 
    Consider the multiarrangement $(\A,m+\bbmo)$ given by $$Q(\A,m+\bbmo)=(y-z)^ay^b(x-y)^cx^d(x-z)^e.$$
    Then $(\A,m+\bbmo)$ is free if and only if $c\geq a+e-1$ or $c\geq b+d-1$.
\end{theorem}

Combining Theorems \ref{theorem: supersolvable exponents} and \ref{theorem: del A3 classification}, we can now prove Theorem \ref{theorem: Multieuler for deleted A3}.

\begin{proof}[Proof of Theorem \ref{theorem: Multieuler for deleted A3}.]
We have the following supersolvable filtrations for $\A$: $$\{x=0\}\subset\{xy(x-y)=0\}\subset\{xy(x-y)(x-z)(y-z)=0\}$$ and $$\{(x-y)=0\}\subset\{(x-y)(x-z)(y-z)=0\}\subset\{xy(x-y)(x-z)(y-z)=0\},$$ which we combine with Theorem \ref{theorem: supersolvable exponents} to calculate the exponents of $(\A,m+\bbmo)$. As in Theorem \ref{theorem: supersolvable exponents}, we denote a chosen supersolvable chain for $\A$ by $\A_1\subset\A_2\subset\A_3$. Theorem \ref{theorem: del A3 classification} shows that $(\A,m+\bbmo)$ is free if and only if $(\A,m+\bbmo)$ is supersolvable with respect to a choice of one of the two given filtrations above. In particular, the simple arrangement $(\A,\bbmo)$ is free with exponents $(1,2,2)$. 

By definition of an $m$-universal derivation $\theta\in D(\A,m+\bbmo)$ the module $D(\A,m)$ has to be free and the exponents of $D(\A,m)$ have to all be equal. Since $D(\A,m)$ is free if and only if it is supersolvable, we can use Theorem \ref{theorem: supersolvable exponents} to calculate $\exp(\A,m)$ and have to demand that $$c-1=a+e-2=b+d-2,$$ else $(\A,m)$ does not have all exponents equal. But this is equivalent to requiring that $c=a+e-1=b+d-1$. It is left to show that this restriction on $m+\bbmo$ is sufficient.

So let $(\A,m+\bbmo)$ be such that $c+1=a+e=b+d$. Using the second part of Theorem \ref{theorem: supersolvable exponents} we deduce that $$\exp(\A,m+\bbmo)=(c,c+1,c+1).$$ Note that $(\A,m)$ and $(\A,m+\bbmo+\delta_H)$ are still supersolvable multiarrangements for an arbitrary $H\in\A$. This follows immediately for $(\A,m)$, since each of the inequalities of Theorem \ref{theorem: supersolvable exponents} is still satisfied if we reduce the multiplicity of every hyperplane in $(\A,m+\bbmo)$ by $1$. For $(\A,m+\bbmo+\delta_H)$ and $H=\ker x$, the inequalities are still satisfied. Finally, if $H\neq \ker x$, then choose a filtration such that $H\in\A_2$, which once again ensures that the inequalities are satisfied. Then apply Theorem \ref{theorem: supersolvable exponents} to derive that $$\exp(\A,m)=(c-1,c-1,c-1)\quad\text{and}\quad\exp(\A,m+\bbmo+\delta_H)=(c+1,c+1,c+1)$$ for all $H\in\A$, so $(\A,m+\bbmo)$ is $c$-critical. Now Theorem \ref{theorem: alternative universal criterion} implies that a basis element $\theta\in D(\A,m+\bbmo)$ with $\deg\theta=c$ is $m$-universal.\end{proof}

\subsection{Braid arrangements}

In \cite{DFMS} DiPasquale, Francisco, Mermin, and Schweig showed that all free multiplicities on the Coxeter arrangement of type $A_3$ are either so called ANN-multiplicities, see \cite[Thm.~0.3]{ANN}, or supersolvable. 
We now determine all $m$-universal derivations for the braid arrangement of rank three in the supersolvable case.

\begin{proposition}\label{proposition: Multieuler for supersolvable A3}
    Let $$Q(\CA,m+\bbmo)=(x_1-x_2)^a(x_1-x_3)^b(x_1-x_4)^c(x_2-x_3)^d(x_2-x_4)^e(x_3-x_4)^f$$ be an $A_3$ multiarrangement.  
    Assume that $(\CA,m+\bbmo)$ satisfies the conditions of Theorem \ref{theorem: supersolvable exponents} for a suitable supersolvable filtration $\CA_1\subset \CA_2\subset \CA_3=\CA$ for $(\CA,m+\bbmo)$. Then there exists an $m$-universal derivation if and only if all of the inequalities in Theorem \ref{theorem: supersolvable exponents} are identities. 
\end{proposition}
\begin{proof}
   To apply Theorem \ref{theorem: supersolvable exponents} we fix the following supersolvable filtration 
   $$\{(x_1-x_2)=0\}\subset\{(x_1-x_2)(x_1-x_3)(x_2-x_3)=0\}$$ $$\subset\{(x_1-x_2)(x_1-x_3)(x_2-x_3)(x_1-x_4)(x_2-x_4)(x_3-x_4)=0\}.$$
   Note, that we can obtain other supersolvable filtrations by permuting the $x_i$ by an element of $S_4$. Recalling that $H_{ij} = \ker(x_i-x_j)$, the conditions of Theorem \ref{theorem: supersolvable exponents} read as follows:
   \begin{align*}
   (m+\bbmo)(H_{12})\geq (m+\bbmo)(H_{14})+(m+\bbmo)(H_{24})-1 &\iff a \geq c+e-1 \text{ and}\\
   (m+\bbmo)(H_{13})\geq (m+\bbmo)(H_{14})+(m+\bbmo)(H_{34})-1 &\iff b \geq c+f-1 \text{ and}\\
   (m+\bbmo)(H_{23})\geq (m+\bbmo)(H_{24})+(m+\bbmo)(H_{34})-1 &\iff d \geq e+f-1.
   \end{align*}
   \noindent Suppose that there exists an $m$-universal derivation. In particular there exists a $d_1\in \mathbb{Z}_{>0}$ such that:
  \begin{enumerate}
      \item $(\CA,m+\bbmo)$ is free with exponents $(d_1+1,d_2,d_3)$, where $d_1+1<d_2,d_3$. 
      \item $(\CA,m)$ is free with exponents $(d_1,d_1,d_1)$.
  \end{enumerate}

    \noindent Condition (2) combined with Theorem \ref{theorem: supersolvable exponents} implies that 
    \begin{equation}\label{eq:necessary}
        \frac{a+b+d-3}{2}=c+e+f-3\iff a+b+d=2c+2e+2f-3
    \end{equation}  has to hold and in particular $a+b+d$ is odd.  
    Thanks to the inequalities above, we have 
    $$a+b+d\geq (c+e-1)+(c+f-1)+(e+f-1)=2c+2e+2f-3.$$
   Combined with (\ref{eq:necessary}), this implies that all of the inequalities required for supersolvability have to be identities. 
    
   It is left to show that (assuming that $(\CA,m+\bbmo)$ is supersolvable) the condition (\ref{eq:necessary}) is not only necessary but sufficient for $(\CA,m)$ to admit an $m$-universal derivation. We do this by showing that for all $H\in\CA$, the multiarrangement $(\CA,m+\bbmo+\delta_H)$ is free and $d_1+1\not\in\exp(\CA,m+\bbmo+\delta_H)$. This then proves that $D(\CA,m+\bbmo)$ is $(d_1+1)$-critical, so Theorem \ref{theorem: alternative universal criterion} implies the existence of an $m$-universal derivation.
   So for the rest of the proof we assume $$a = c+e-1, b = c+f-1, \text{ and } d = e+f-1.$$
   Since all multiplicities are positive, this shows $a+b>d, a+d>b, b+d>a$  and we can use Theorem \ref{theorem: Wakamiko A2 classification} to calculate $\exp(\CA_2,m_2)$ and get $\exp(\CA_2,m_2)=\left(\frac{a+b+d-1}{2},\frac{a+b+d+1}{2}\right)$. Theorem \ref{theorem: supersolvable exponents} yields $$\exp(\CA,m+\bbmo)=\left(\frac{a+b+d-1}{2},\frac{a+b+d+1}{2},\frac{a+b+d+3}{2}\right)$$ and $\exp(\CA,m)=(\frac{a+b+d-3}{2},\frac{a+b+d-3}{2},\frac{a+b+d-3}{2})$.\\
   It is left to show that $\exp(\CA,m+\bbmo+\delta_H)=(\frac{a+b+d+1}{2},\frac{a+b+d+1}{2},\frac{a+b+d+3}{2})$ for all $H\in\CA$.\\
   If $H\in\{\ker(x_1-x_2),\ker(x_1-x_3),\ker(x_2-x_3)\}$, then $(\CA,m+\bbmo+\delta_H)$ is still supersolvable since all of the inequalities are still satisfied and Theorem \ref{theorem: supersolvable exponents} implies $\exp(\CA,m+\bbmo+\delta_H)=(\frac{a+b+d+1}{2},\frac{a+b+d+1}{2},\frac{a+b+d+3}{2})$.\\
   Now let $H\in\{\ker(x_1-x_4),\ker(x_2-x_4),\ker(x_3-x_4)\}, H'\in\CA\backslash\{H\}, X=H\cap H'$.
   \noindent We make use of Theorem \ref{theorem: Wakamiko A2 classification} to derive the following concerning the Euler restriction $(m+\bbmo+\delta_H)^*$ with respect to $H$ on $\A$:
   \begin{enumerate}
    \item If $\CA_X=\{H,H'\}$, then we 
    have $(m+\bbmo+\delta_H)^*(X)=(m+\bbmo)(H')$. 
    \item If $\CA_X=\{(x_1-x_2),(x_1-x_4),(x_2-x_4)\}$, then use the equality $a=c+e-1$ to see that $\exp(\CA_X,(m+\bbmo)_X)=(a,c+e)=(a,a+1)$, $\exp(\CA_X,(m+\bbmo+\delta_H)_X)=(a+1,a+1)$ and therefore $(m+\bbmo+\delta_H)^*(X)=a+1=c+e$ is independent from the choice of $H$.
    \item Analogous, if $\CA_X=\{(x_1-x_3),(x_1-x_4),(x_3-x_4)\}$, then use the identity $b=c+f-1$ to see that $(m+\bbmo+\delta_H)^*(X)=b+1=c+f$ and if $\CA_X=\{(x_2-x_3),(x_2-x_4),(x_3-x_4)\}$, then use $d=e+f-1$ to infer that $(m+\bbmo+\delta_H)^*(X)=d+1=e+f$.
   \end{enumerate}

   \noindent Now let $H=\ker(x_1-x_4)$. Then we have the following three localizations $\CA_X$: $$\{H,\ker(x_1-x_2),\ker(x_2-x_4)\},\{H,\ker(x_1-x_3),\ker(x_3-x_4)\},\{H,\ker(x_2-x_3)\}.$$
   Fixing this order of localizations we have $(m+\bbmo)_X=(c,a,e),(c,b,f),(c,d)$. Now use the results of the discussion above to see that $(m+\bbmo+\delta_H)^*=(a+1,b+1,d)$. So since $\vert(m+\bbmo+\delta_H)^*\vert=a+b+d+2$, the sum $a+b+d$ is an odd number and the multiplicity is balanced on $(\CA_2,m_2)$. Use Theorem \ref{theorem: Wakamiko A2 classification} once more to calculate $$\exp(\CA^H,(m+\bbmo+\delta_H)^*)=\left(\frac{a+b+d+1}{2},\frac{a+b+d+3}{2}\right),$$ as required.
   
   \noindent The calculations for all remaining hyperplanes $H\in\CA$ work analogously. This proves the existence of an $m$-universal derivation. The arguments for the other supersolvable filtrations work the same way. That such a multiplicity $m+\bbmo$ which meets the requirements of the theorem exists is obvious. 
   \end{proof}

Proposition \ref{proposition: Multieuler for supersolvable A3} gives Theorem \ref{A3multiEuler}. This  provides examples of the following kind.

\begin{example}
\label{twocomponent}
Let $\A$ be defined as 
$$
Q(\A)\coloneqq xyz(x-y)(y-z)(x-z).
$$
Then it is well-known that 
$$
Q(\A,m+\bbmo)\coloneqq x^3y^3z^3(x-y)^3(y-z)^3(x-z)^3
$$
is free with $\exp(\A,m+\bbmo)=(5,6,7)$, see \cite{T02}. Hence using the Addition-Deletion-Theorem \cite[Thm.~0.8]{ATW2}, we obtain an $m$-universal non-zero homogeneous basis element $\theta_1 \in D(\A,m+\bbmo)$ with $\deg \theta_1=5$. Now the definition of $m$-universal derivations affords a free basis. For $$Q(\A,m_1+\bbmo)\coloneqq x^3y^3z^2(x-y)^3(y-z)^2(x-z)^2,
$$
we have $\exp (\A,m_1+\bbmo) = (4,5,6)$. Now use the Addition-Deletion-Theorem \cite[Thm.~0.8]{ATW2} once again, to confirm that a non-zero homogeneous basis element $\varphi_1 \in D(\A,m_1+\bbmo)$ of degree $4$ is $m_1$-universal. In fact, for 
$$
\varphi_1=(x^4-2x^3y)\partial_x+
(-2xy^3+y^4)\partial_y+
(-3z^4-6xyz^2+4xz^3+4yz^3)\partial_z,
$$
we obtain
\begin{eqnarray*}
\nabla_{\partial_x}\varphi_1&=&(4x^3-6x^2y)\partial_x-2y^3\partial_y+(-6yz^2+4z^3)\partial_z,\\
\nabla_{\partial_y}\varphi_1&=&-2x^3\partial_x+(-6xy^2+4y^3)\partial_y+(-6xz^2+4z^3)\partial_z,\\
\nabla_{\partial_z}\varphi_1&=&-12z(z-x)(z-y)\partial_z.
\end{eqnarray*}
So Saito's criterion entails that the last three derviations form a basis for $D(\A,m_1)$, where 
$$
Q(\A,m_1)\coloneqq x^2y^2z(x-y)^2(y-z)(x-z).
$$
So we can directly check that $\varphi_1$ is $m_1$-universal. Define the constant multiplicity 
$\mathbbm{2}$ on $\A$ by $\mathbbm{2}(H) = 2$ for each $H$ in $\A$.  There are free multiarrangements like $(\A,\mathbbm{2})$ admitting bases which stem from two different sets of $m$-universal derivations as follows:
$$
\langle 
\nabla_{\partial_{x}} \theta_1,\nabla_{\partial_{y}} \theta_1,\nabla_{\partial_{z}} \theta_1
\rangle_S=D(\A,\mathbbm{2}),
$$
as well as,
$$
\langle 
\nabla_{\psi_1} \varphi_1,\nabla_{\partial_{\psi_2}} \varphi_1,\nabla_{\psi_3} \varphi_1
\rangle_S=D(\A,\mathbbm{2}),
$$
where $\psi_1,\psi_2,\psi_3$ form a basis of $D(\B)$ for the subarrangement $\B$ of $\A$ defined by 
$$
Q(\B)=z(y-z)(x-z).
$$
\end{example}

\subsection{$X_3$}
Now we examine the $X_3$ arrangement given by $
Q(\A)=xyz(x+y)(y+z)(x+z)
$ and prove Theorem \ref{theorem: classification universal for X3}. The latter states that $(\A,m)$ does not admit an $m$-universal derivation for any $m$ distinct from the Euler derivation.

\begin{proof}[Proof of Theorem \ref{theorem: classification universal for X3}.]
Owing to \cite{DW}, the multiarrangement $(\A,m)$ is free if and only if 
$$
Q(\A,m)=x^{2n}y^{2n}z^{2n}(x+y)(y+z)(x+z)
$$
for a non-negative integer $n$ and in that case $\exp(\A,m)=(2n+1,2n+1,2n+1)$. Thus, if there exists an $m$-universal derivation, then it belongs to $D(\A,m+\bbmo)$, where 
$$Q(\A,m+\bbmo)=x^{2n+1}y^{2n+1}z^{2n+1}(x+y)^2(y+z)^2(x+z)^2.$$
Subsequently, such an $m$-universal derivation $\theta$ is of the form 
$$
\theta=x^{2n+1}(a_1x+b_1y+c_1z)\partial_x+
y^{2n+1}(a_2x+b_2y+c_2z)\partial_y+
z^{2n+1}(a_3x+b_3y+c_3z)\partial_z,
$$
for $a_i, b_i, c_i \in \K$.
However, an easy computation shows that such a $\theta$ has to satisfy $a_i=b_i=c_i=0$. So $\theta$ satisfies $m$-universality only for $m\equiv 0$, i.e., $\theta=\theta_E$.
\end{proof}

\begin{problem}\label{problem: not free euler}
If $\A$ is not free, then does there exist any multiplicity $m$ on $\A$ which admits a non-Euler $m$-universal derivation $\theta \in D(\A,m+\bbmo)$?
\end{problem}

\subsection{Supersolvable multiarrangements, II}
In this section we investigate supersolvable multiarrangements $(\CA,m+\bbmo)$, where $\rank(\CA)=3$. Fix a supersolvable filtration $$(\CA_1,m_1+\bbmo)\subset(\CA_2,m_2+\bbmo)\subset(\CA_3,m+\bbmo)=(\CA,m+\bbmo)$$ of $(\CA,m+\bbmo)$. First we assume that $\CA_2=\CA(A_2)$ and that $m+\bbmo$ is balanced. The following three cases are to be considered.\\
1. If $\vert \CA\backslash\CA_2\vert=2$, then $\CA$ is the deleted $A_3$ arrangement.\\
2. If $\vert \CA\backslash\CA_2\vert=3$, then $\CA$ is the Coxeter arrangement of type $A_3$.\\
3. If $\vert \CA\backslash\CA_2\vert >3$, then $$Q(\CA)=xy(x-y)\prod_{1\leq i\leq \vert\A\vert-3}(z-a_ix).$$
Note that only the last case needs to be considered due to the results in the earlier sections. For this we require the following result from \cite{ATW2}.

\begin{lemma}[{\cite[Lem.~3.4]{ATW2}}]
\label{lemma: ATW B Polynomial}
    Let $(\CA,m+\bbmo)$ be a multiarrangement. Fix $H_0\in\CA$ and let $m_0\coloneqq(m+\bbmo+\delta_{H_0})(H_0)$. For every $X\in\CA'' := \A^{H_0}$ fix an $H_X\in\CA\backslash\{H_0\}$ such that $X=H_0\cap H_X$ and define $d_X\in\exp(\CA_X,m_X+\bbmo+\delta_H)$ as the unique non-shared exponent of $(\CA_X,m_X+\bbmo)$ and $(\CA_X,m_X+\bbmo+\delta_{H_0})$. Define the polynomial $B=B(\CA'',(m+\bbmo+\delta_{H_0})^*)$ by $$B=\alpha_{H_0}^{m_0-1}\prod_{X\in\CA''}\alpha_{H_X}^{d_X-m_0}.$$
    For any $\theta\in D(\CA,m+\bbmo)$ we have $\theta(\alpha_0)\in(\alpha_0^{m_0},B)$.
\end{lemma}

From Lemma \ref{lemma: ATW B Polynomial} we can derive the following criterion for the non-existence of universal derivations for a multiarrangement $(\CA,m+\bbmo)$.

\begin{corollary}\label{corollary: Bpolynomial corollary}
    With the notation as in Lemma \ref{lemma: ATW B Polynomial} let $\theta\in D(\CA,m+\bbmo)$ be homogeneous with $\deg \theta <\deg B$. Then $\theta\in D(\CA,m+\bbmo+\delta_{H_0})$ and in particular, $D(\CA,m+\bbmo+\delta_{H_0})_{\deg\theta}\neq (0)$.
\end{corollary}
\begin{proof}
    Thanks to Lemma \ref{lemma: ATW B Polynomial} we have $\theta(\alpha_0)\in(\alpha_0^{m_0},B)$, but from $\deg(\theta)<\deg(B)$ and $\deg(\alpha_0)=1$, we derive $\deg(\theta(\alpha_0))<\deg(B)$. This shows that we have $\theta(\alpha_0)\in \alpha_0^{m_0}\cdot S$ and therefore $\theta\in D(\CA,m+\bbmo+\delta_{H_0})$.
\end{proof}

\begin{lemma}\label{lemma: not critical criteria}
    Let $(\CA,m+\bbmo)$ be a free multiarrangement with $\rank(\CA)=3$ and exponents $\exp(\CA,m+\bbmo)=(d_1,d_2,d_3)$, where $d_1\leq d_2\leq d_3$. If there exists an $H\in\CA$ such that $\vert(\CA^H,(m+\bbmo+\delta_H)^*)\vert<d_2+d_3$, then $D(\CA,m+\bbmo+\delta_H)_{d_1}\neq(0)$.  
\end{lemma}
\begin{proof}
    Since $\exp(\CA,m+\bbmo)=(d_1,d_2,d_3)$, we have $D(A,m+\bbmo)_{d_1}\neq (0)$. Due to Corollary \ref{corollary: Bpolynomial corollary} it is sufficient to show $d_1<\deg B$. By definition of $B$ and substituting $m_0=(m+\bbmo+\delta_H)(H)$, we have
    $$\deg  B=(m+\bbmo)(H)+\sum_{X\in\CA^H}\left(d_X-m_0\right),$$
    where $d_X$ and $B$ are as in Lemma \ref{lemma: ATW B Polynomial}. By the definitions of the Euler multiplicity (Definition \ref{def:Euler}) and $d_X$ we have $\vert(m+\bbmo+\delta_H)_X\vert=d_X+(m+\bbmo+\delta_H)^*(X)$. Therefore,  
    \begin{align*}
    \sum_{X\in\CA^H}d_X 
    & =\left(\sum_{X\in\CA^H}\vert(m+\bbmo+\delta_H)_X\vert\right)-\vert(m+\bbmo+\delta_H)^*\vert.
      \end{align*}
    Now utilize these equations to derive
    $$\deg B - d_1 = (m+\bbmo)(H)+\sum_{X\in\CA^H}\left(d_X-m_0\right)-d_1$$
    $$=(m+\bbmo)(H)+\left(\sum_{X\in\CA^H}d_X\right)-\vert\CA^H\vert\cdot m_0-d_1$$
    $$=(m+\bbmo)(H)+\left(\sum_{X\in\CA^H}\vert (m+\bbmo+\delta_H)_X\vert \right)-\vert(m+\bbmo+\delta_H)^*\vert-\vert\CA^H\vert\cdot m_0-d_1$$
    $$=(m+\bbmo)(H)+\sum_{X\in\CA^H}\left(\vert (m+\bbmo+\delta_H)_X\vert - m_0\right)-\vert(m+\bbmo+\delta_H)^*\vert-d_1.$$
    Finally, utilize $\vert (m+\bbmo)\vert=d_1+d_2+d_3$ and $\vert(m+\bbmo+\delta_H)^*\vert<d_2+d_3$ to derive
    $$(m+\bbmo)(H)+\sum_{X\in\CA^H}\left(\vert (m+\bbmo+\delta_H)_X\vert - m_0\right)-\left(\vert(m+\bbmo+\delta_H)^*\vert+d_1\right)$$ 
    $$>(m+\bbmo)(H)+\sum_{X\in\CA^H}\left(\vert (m+\bbmo+\delta_H)_X\vert - (m+\bbmo+\delta_H)(H)\right)-\vert (m+\bbmo)\vert=0.$$
    Consequently, we have shown $d_1<\deg B$, as desired.
\end{proof}

Next, we study the case when $\A$ is  supersolvable with $\CA_2=\CA(A_2)$ and $|\A\backslash\A_2|>3$.

\begin{proposition}
    Let $Q(\CA)=xy(x-y)\prod_{1\leq i\leq h}(z-a_ix)$ with $h\coloneqq|\A\backslash\A_2|>3$ and $m+\bbmo$ a multiplicity on $\CA$ such that $(\CA,m+\bbmo)$ is supersolvable for the filtration $\{\ker x\}\subset\A_2 = \CA(A_2)\subset\A_3 = \CA.$ Then there does not exist any universal derivation for $(\CA,m+\bbmo)$.
\end{proposition}
\begin{proof}
    Since $(\CA,m+\bbmo)$ is supersolvable, we utilize Theorems \ref{theorem: Wakamiko A2 classification} and \ref{theorem: supersolvable exponents} to derive 
    
    $$\exp(\CA,m+\bbmo)=\left(\left\lfloor\frac{\vert m_2+\bbmo\vert}{2}\right\rfloor,\left\lceil\frac{\vert m_2+\bbmo\vert}{2}\right\rceil,\vert m+\bbmo\vert-\vert m_2+\bbmo\vert\right).$$ 
    
    If there exists a universal derivation, then all exponents of $(\CA,m)$ have to be equal. This shows that $m_2+\bbmo$ must be a peak point for $\CA(A_2)$ and therefore $\vert m_2+\bbmo\vert$ is odd. This condition on $\exp(\CA,m)$ also forces $\frac{\vert m_2+\bbmo\vert-3}{2}=\vert m+\bbmo\vert-\vert m_2+\bbmo\vert-h$. Hence

    $$\exp(\CA,m+\bbmo)=\left(\frac{\vert m_2+\bbmo\vert-1}{2},\frac{\vert m_2+\bbmo\vert+1}{2},\frac{\vert m_2+\bbmo\vert-3}{2}+h\right).$$
    
    Let $H\in\CA_3\backslash\CA_2$. Due to Theorem \ref{theorem: alternative universal criterion} and Lemma \ref{lemma: not critical criteria} it suffices to show 
    
    $$\vert (m+\bbmo+\delta_H)^* \vert < \frac{\vert m_2+\bbmo\vert+1}{2} + \frac{\vert m_2+\bbmo\vert-3}{2}+h = \vert m_2+\bbmo\vert+h-1.$$
    
    We have $\vert (m+\bbmo+\delta_H)^*\vert=\vert m_2+\bbmo\vert$ since for all localizations of rank two we either have $(m+\bbmo)(\ker x)\geq (\sum_{H\in(\CA\backslash\CA_2)}(m+\bbmo)(H))-1$ due to the supersolvability of $(\CA,m+\bbmo)$ or $\vert \CA_X\vert=2$. Since $h\geq 3$, we derive $\vert (m+\bbmo+\delta_H)^*\vert=\vert m_2+\bbmo\vert < \vert m_2+\bbmo\vert+h-1$ as desired.    
\end{proof}

In our final result we investigate a supersolvable arrangement $(\CA,m+\bbmo)$, where $\rank(\CA)=3$ and $\CA_2=\CA(B_2)$ in the supersolvable filtration $$(\CA_1,m_1+\bbmo)\subset(\CA_2,m_2+\bbmo)\subset(\CA_3,m+\bbmo)=(\CA,m+\bbmo).$$
\begin{proposition}
Let $\CA=\CA(B_3)$ be the Coxeter arrangement of type $B_3$ 
with supersolvable filtration $$\{\ker x\}\subset\{\ker x,\ker y,\ker(x+y),\ker(2x+y)\}\subset \CA.$$
Let $$Q(\CA,m+\bbmo)=x^ay^b(x+y)^c(2x+y)^dz^e(y+z)^f(x+y+z)^g(2x+y+z)^h(2x+2y+z)^i.$$ 
Then there does not exist a universal derivation for $(\CA,m+\bbmo)$.
\end{proposition}
\begin{proof}
\noindent Let $m+\bbmo$ such that $(\CA,m+\bbmo)$ is supersolvable with the showcased filtration. Then, the following inequalities need to be satisfied:\\
$a\geq f+g+h-1$, $b\geq e+f-1$, $b\geq h+i-1$, $c\geq e+g+i-1$, $d\geq e+h-1$, and $d\geq f+i-1$.\\
We have $\vert m_2+\bbmo\vert=a+b+c+d$ and $\vert m+\bbmo\vert- \vert m_2+\bbmo\vert=e+f+g+h+i$. From the inequalities above, we derive
\begin{itemize}
    \item $a+b+c+d \geq (f+g+h-1)+(e+f-1)+(e+g+i-1)+(e+h-1)$\\
    $\iff a+b+c+d \geq 3 e + 2 f + 2 g + 2 h + i - 4$\\
    $\iff\frac{a+b+c+d}{2}+2 \geq \frac{3 e + i}{2}+f+g+h$.\\
    \item $a+b+c+d \geq (f+g+h-1)+(h+i-1)+(e+g+i-1)+(f+i-1)$\\
    $\iff a+b+c+d \geq  e + 2 f + 2 g + 2 h + 3 i -4$\\
    $\iff\frac{a+b+c+d}{2}+2 \geq \frac{3 i + e}{2}+f+g+h$.\\
\end{itemize}
If $e\geq i$, then use the first inequality and derive
$$\frac{a+b+c+d}{2}+2\geq \frac{3 e + i}{2}+f+g+h \geq e+f+g+h+i. $$
If $e\leq i$, then use the second inequality and derive
$$\frac{a+b+c+d}{2}+2\geq \frac{e + 3 i}{2}+f+g+h \geq e+f+g+h+i. $$
If $m$ had all equal exponents, we would have 
$$\frac{a+b+c+d-4}{2}=e+f+g+h+i-5$$
$$\iff \frac{a+b+c+d}{2}+2=e+f+g+h+i-1.$$ This contradicts $$\frac{a+b+c+d}{2}+2 \geq e+f+g+h+i.$$
This completes the proof.
\end{proof}

\subsection{Totally free arrangements and universal derivations}
We finish this section by presenting a totally free arrangement for which no non-trivial universal derivation exists. 

\begin{proof}[Proof of Theorem \ref{theorem: freeness does not imply universal}]
    Let $\A$ be the real arrangement given by $Q(\CA)=xy(x-y)(x-\pi y)$. Maehara introduced this arrangement in \cite[Cor.~3.10]{Ma} and showed that for a balanced multiplicity $m\not = \bbmo$ on $\CA$ we have \begin{enumerate}
        \item $\Delta(m)=0$, if $\vert m\vert$ is even, and 
        \item $\Delta(m)=1$, if $\vert m\vert$ is odd.
    \end{enumerate}
    Since $\rank(\CA)=2$ the multiarrangement $(\A,m)$ is totally free. As $\vert \CA\vert=4$, Theorem \ref{theorem: rank two universal criterion} shows that there does not exist a universal derivation (other than $\theta_E$) for $\CA$.    
\end{proof}

\section{Open problems}
In addition to Problem \ref{problem: not free euler} we list further problems in this section. We derived new examples of universal derivations through our new criteria, namely Theorems \ref{theorem: alternative universal criterion} and \ref{theorem: rank two universal criterion}. However, these results still require knowledge of the exponents.  
\begin{problem}
    Is there a criterion that guarantees the existence of universal derivations for a given multiarrangement $(\CA,m)$ without a priori knowledge of $\exp(\CA,m)$?
\end{problem}

It was shown by Terao in \cite{T80} that the freeness of $(\CA,m)$ implies the freeness of $(\CA_X,m_X)$ for an arbitrary $X\in L(\CA)$. This motivates the following question.

\begin{problem}
    Does the existence of a universal derivation $\theta\in D(\CA,m)$ imply the existence of a universal derivation $\theta_X\in D(\CA_X,m_X)$?
\end{problem}

As it turns out for numerous supersolvable arrangements $\CA$ one can define a multiplicity $m$ such that $(\CA,m)$ is supersolvable and admits a universal derivation. This suggests that supersolvability is a requirement for an universal derivation to exist. We emphasize that this is not the case, i.e., see  \cite{ATW} and \cite{AY}.

\medskip

\bibliographystyle{amsalpha}

\begin{thebibliography}{Z}
 \bibitem{A1} T. Abe, 
 The stability of the family of $A_2$-type arrangements. \textit{J. Math. Kyoto Univ.} 
 \textbf{46} (2006),  no. 3, 617--639.

\bibitem{A07} T. Abe,
Free and non-free multiplicity on the deleted $A_3$ arrangement. \textit{Proceedings of the Japan Academy Series A: Mathematical Sciences}, \textbf{83} (2007), No. 7, 99-103.

\bibitem{A5} T. Abe, 
Chambers of $2$-affine arrangements and freeness of $3$-arrangements 
\textit{J. Alg. Combin.} 
\textbf{38} (2013), no.1,  65--78.

 \bibitem{ANN} T. Abe, K. Nuida and Y. Numata, 
 Signed-eliminable graphs and free multiplicities on the braid arrangement. 
\textit{J. London Math. Soc.} 
\textbf{80} (2009), no. 1, 121--134.

\bibitem{AN} T. Abe and Y. Numata, 
 Exponents of $2$-multiarrangements and multiplicity lattices. 
 \textit{J. Alg. Combin.} 
\textbf{35} (2012), no.1,  1--17.

\bibitem{ASRY} T. Abe, C. Stump, G. Röhrle, and M. Yoshinaga, 
A Hodge filtration of logarithmic vector fields for well-generated complex reflection groups, J. Comb. Algebra {\bf 8} (2024), no.~3-4, 251--278; 
  
\bibitem{ATW} T. Abe, H. Terao and A. Wakamiko, Equivariant multiplicities 
of Coxeter arrangements and 
invariant bases. 
\textit{Adv. Math}. \textbf{230} (2012), no. 4--6, 2364--2377.

 \bibitem{ATW2} T. Abe, H. Terao and M. Wakefield, The Euler multiplicity and 
addition-deletion theorems for multiarrangements. 
\textit{J. London Math. Soc.} \textbf{77} (2008), no. 2, 335--348.

 \bibitem{AY} T. Abe and M. Yoshinaga, Coxeter multiarrangements with quasi-constant 
 multiplicities. \textit{J. Algebra} \textbf{322} (2009), no. 8, 2839--2847.

\bibitem{DFMS}
M. DiPasquale, 
C. Francisco, J. Mermin, and J. Schweig, 
Free and non-free multiplicities on the A3 arrangement. 
\textit{J. Algebra} 
\textbf{544} (2020), 498--532. 

\bibitem{DW}
M. DiPasquale and M. Wakefield,
Free multiplicities on the moduli of $X_3$, 
\textit{J. Pure Appl. Algebra} 
\textbf{222} (2018), no. 11, 3345--3359.

\bibitem{Ma} S. Maehara, \textit{Exponents of $2$-multiarrangements and Wakefield-Yuzvinsky matrices}, \href{https://arxiv.org/abs/2509.16569}{\tt arXiv:2509.16569}.

 \bibitem{OT} P. Orlik and H. Terao, \textit{Arrangements of hyperplanes}.
 Grundlehren der Mathematischen Wissenschaften, 
 \textbf{300}. Springer-Verlag, Berlin, 1992.

 \bibitem{Sa} {K. Saito},
 Theory of logarithmic differential forms and logarithmic vector fields. 
 \textit{J. Fac. Sci. Univ. Tokyo Sect. IA  Math}. 
 \textbf{27} (1980), 265--291. 

 \bibitem{T80} H. Terao, Arrangements of hyperplanes and their freeness I. \textit{J. Fac. Sci. Univ. Tokyo Sect. IA Math.}
 \textbf{27} (1980), 293--320. 

\bibitem{T02} H. Terao, Multiderivations of Coxeter arrangements. \textit{Invent. Math. }
 \textbf{148} (2002), 659--674. 

\bibitem{W07} A. Wakamiko, On the Exponents of $2$-Multiarrangements. 
\textit{Tokyo J. Math.} \textbf{30} (2007), no. 1, 99--116.

\bibitem{W10} A. Wakamiko, Bases for the derivation modules of two-dimensional multi-Coxeter arrangements and universal derivations. \textit{Hokkaido Math. J.} \textbf{40} (2011), no. 3, 375--392.

 \bibitem{Y0} M. Yoshinaga, The primitive derivation and freeness of multi-Coxeter arrangements. 
 \textit{Proc. Japan Acad. Ser. A} \textbf{78} (2002), no. 7,  116--119.

 \bibitem{Y2} M. Yoshinaga, Characterization of a free arrangement and
 conjecture of
 Edelman and Reiner. \textit{Invent. Math.} \textbf{157} (2004), no. 2,
 449--454.

 \bibitem{Y3} M. Yoshinaga, On the freeness of 3-arrangements. 
 \textit{Bull. London Math. Soc.} \textbf{37} (2005), no. 1, 126--134. 

 \bibitem{Z}
 G. M. Ziegler, 
 Multiarrangements of hyperplanes and their freeness. in
 {\it Singularities} (Iowa City,
 IA, 1986), 345--359, Contemp. Math., {\bf 90}, Amer. Math. Soc.,
 Providence, RI, 1989.

 
\bibitem{Z2}
G. M. Ziegler, Combinatorial construction of logarithmic differential forms.
\textit{Adv. Math}. \textbf{76} (1989), 116--154.
\end{thebibliography}
\newcommand{\etalchar}[1]{$^{#1}$}
\providecommand{\bysame}{\leavevmode\hbox to3em{\hrulefill}\thinspace}
\providecommand{\MR}{\relax\ifhmode\unskip\space\fi MR }

\end{document}